\documentclass[openany,a4paper,12pt]{article}
\usepackage[utf8]{inputenc}
\usepackage[T1]{fontenc}
\usepackage[english]{babel}
\usepackage[intlimits]{amsmath}
\usepackage[]{graphicx}
\usepackage{color}
\usepackage{hyperref} 
\usepackage{amssymb}
\usepackage{amsmath}
\usepackage{array}
\usepackage{geometry}
\usepackage{pdfpages}

\newcommand{\K}{\mathbb{K}}

\newtheorem{defi}{Definition}
\newtheorem{con}[defi]{Conjecture}
\newtheorem{teor}[defi]{Theorem}
\newtheorem{lema}[defi]{Lemma}

\newtheorem{prp}[defi]{Proposition}

\newenvironment{proof}{{\it Proof:}}{\hspace{\stretch{1}}\rule{1ex}{1ex}}

\usepackage{titlesec}

\titleformat*{\section}{\normalsize\bfseries}
\titleformat*{\subsection}{\normalsize\bfseries}
\titleformat*{\subsubsection}{\normalsize\bfseries}
\titleformat*{\paragraph}{\normalsize\bfseries}
\titleformat*{\subparagraph}{\normalsize\bfseries}

\usepackage{setspace}
\begin{document}

\begin{center}
{\normalsize \bf IMAGES OF MULTILINEAR POLYNOMIALS OF DEGREE UP TO FOUR ON UPPER TRIANGULAR MATRICES}\\
{\normalsize Pedro S. Fagundes}\footnote{pedrosfmath@gmail.com\\ Universidade Federal de São Paulo, Instituto de Ciência de Tecnologia, SP, Brazil\\  }{\normalsize and Thiago C. de Mello}\footnote{tcmello@unifesp.br\\ Universidade Federal de São Paulo, Instituto de Ciência de Tecnologia, SP, Brazil} 
\end{center}

\begin{abstract}
We describe the images of multilinear polynomials of degree up to four on the upper triangular matrix algebra. 

\end{abstract}

\section{Introduction}

A famous open problem known as Lvov-Kaplansky's conjecture asserts: the image of a multilinear polynomial in noncommutative variables over a field $\K$ on the matrix algebra $M_{n}(\K)$ is always a vector space \cite{Dniester}.

Recently, Kanel-Belov, Malev and Rowen \cite{Kanel2} made a major breakthrough and solved the problem for $n=2$.

A special case on polynomials of degree two has been known for long time (\cite{Shoda} and \cite{Albert}). Recently, Mesyan \cite{Mesyan} and Buzinski and Winstanley \cite{Buzinski} extended  this result for nonzero multilinear polynomials of degree three and four, respectively.

We will study the following variation of the Lvov-Kaplansky's conjecture:

\begin{con}\label{c1}
The image of a multilinear polynomial on the upper triangular matrix algebra is a vector space.
\end{con}

In this paper, we will answer Conjecture \ref{c1} for polynomials of degree up to four. We point out that whereas in \cite{Buzinski} and \cite{Mesyan} the results describe conditions under which the image of a multilinear polynomial $p$, $Im(p)$, contains a certain subset of $M_n(\K)$, our results give the explicit forms of $Im(p)$ on the upper triangular matrix algebra in each case.

Throughout the paper $UT_{n}$ will denote the set of upper triangular matrices. The set of all strictly upper triangular matrices will be denoted by $UT_{n}^{(0)}$. More generally, if $k\geq 0$, the set of all matrices in $UT_{n}$ whose entries $(i,j)$ are zero, for $j-i\leq k$, will be denoted by $UT_{n}^{(k)}$.  Also if $i,j\in \{1,\dots,n\}$, we denote by $e_{ij}$ the $n\times n$ matrix with 1 in the entry $(i,j)$, and $0$ elsewhere. These will be called matrix units. In particular, $UT_n^{(k)}$ is the vector space spanned by the $e_{ij}$ with $j-i>k$.

Our main goal in this paper is to prove the following:

\begin{teor}
Let $n\geq2$ be an integer.
\begin{itemize}
\item[$(1)$] If $\K$ is an any field and $p$ is a multilinear polynomial over $\K$ of degree two, then $Im(p)$ over $UT_{n}$ is $\{0\}, UT_{n}^{(0)}$ or $UT_{n}$;
\item[$(2)$] If $\K$ is a field with at least $n$ elements and $p$ is a multilinear polynomial over $\K$ of degree three, then $Im(p)$ over $UT_{n}$ is $\{0\}, UT_{n}^{(0)}$ or $UT_{n}$;
\item[$(3)$] If $\K$ is a zero characteristic field and $p$ is a multilinear polynomial over $\K$ of degree four, then $Im(p)$ over $UT_{n}$ is $\{0\}, UT_{n}^{(1)}, UT_{n}^{(0)}$ or $UT_{n}$.
\end{itemize}
\end{teor}

To prove the statement $(1)$ we use some ideas of Shoda \cite{Shoda} and Albert and Muckenhoupt \cite{Albert}, and for statements $(2)$ and $(3)$ we use the polynomial reductions of Mesyan \cite{Mesyan}, \v{S}penko \cite{Spela} and Buzinski and Winstanley \cite{Buzinski}.

\section{The linear span of a multilinear polynomial on $UT_{n}$}

Throughout this section we will denote by $\K$ an arbitrary field and by $p(x_{1},\dots,x_{m})$ a multilinear polynomial in $\K\langle X \rangle$.  We will also denote by $\langle Im(p) \rangle$ the linear span of $Im(p)$ on $UT_{n}$.

We start with an analogous result to Lemma $4$ from \cite{Kanel2}, where we analyse the image of a multilinear polynomial $p(x_{1},\dots,x_{m})\in\K\langle X \rangle$ on upper triangular matrix units.

Let $e_{i_{1},j_{1}},\dots,e_{i_{m},j_{m}}$ be upper triangular matrix units. Then $i_{q}\leq j_{q}$ for all $q$. We know that 
\begin{eqnarray}\label{e1}
e_{i_{1},j_{1}}\cdots e_{i_{m},j_{m}}
\end{eqnarray}
is nonzero (and equal to $e_{i_{1},j_{m}}$) if and only if $j_{q}=i_{q+1}$, for all $q$.

Hence, if we change the order of the product in (\ref{e1}) we will obtain either $0$ or $e_{i_{1},j_{m}}$. To verify this claim, we will assume that we get a nonzero matrix unit after changing the order of some terms in (\ref{e1}). It is enough analyse just when we change the first or the last term. So, if $e_{i_{k},j_{k}}\cdots e_{i_{1},j_{1}}\cdots e_{i_{m},j_{m}}$ is nonzero then $i_{k}\leq i_{1}$, and by the product (\ref{e1}) we also have $i_{1}\leq i_{k}$, which proves that $i_{k}=i_{1}$ and therefore $e_{i_{k},j_{k}}\cdots e_{i_{1},j_{1}}\cdots e_{i_{m},j_{m}}=e_{i_{1},j_{m}}$. Analogously we prove that if $e_{i_{1},j_{1}}\cdots e_{i_{m},j_{m}}\cdots e_{i_{k},j_{k}}$ is nonzero then this product will be $e_{i_{1},j_{m}}$.

In this way, $p$ evaluated on upper triangular matrix units is equal to zero or to some multiple of an upper triangular matrix unit.

\begin{defi}
	Let $A=\displaystyle\sum_{i,j=1}^{n}a_{i,j}e_{i,j}\in UT_{n}$. For each $k\in\{1,\dots,n\}$ the $k$-th diagonal of $A$ is the one with entries in positions $(1,k),(2,k+1),\dots,(n-k+1,n)$. We say that the $k$-th diagonal of $A$ is nonzero if at least one entry in its $k$-th diagonal is nonzero.
\end{defi}

The next lemma shows that if an upper triangular matrix unit can be obtained as an evaluation of a multilinear polynomial on matrix units, then all matrix units in the same diagonal can also be obtained by one such evaluation.

\begin{lema}\label{l2}
Assume that a nonzero multiple of $e_{i,i+k-1}$ can be written as an evaluation of $p$ on upper triangular matrix units, for some $i$ and $k$. Then $e_{1,k},e_{2,k+1},\dots,e_{n-k+1,n}\in Im(p)$.
\end{lema}

\begin{proof}
We write $\alpha e_{i,i+k-1}=p(e_{i_{1},j_{1}},\dots,e_{i_{m},j_{m}})$, for some nonzero $\alpha\in\K$. Hence,
$$\alpha e_{1,k}=p(e_{i_{1}-i+1,j_{1}-i+1},\dots,e_{i_{m}-i+1,j_{m}-i+1}),$$
and since $Im(p)$ is closed under scalar multiplication, $e_{1,k}\in Im(p)$. Analogously, we prove that $e_{2,k+1},\dots,e_{n-k+1,n}\in Im(p)$.
\end{proof}

\begin{lema}\label{l4}
Assume that a nonzero multiple of $e_{i,i+k-1}$ can be written as an evaluation of $p$ on upper triangular matrix units, for some $i$ and $k$. Then $e_{i,i+k}\in Im(p)$.
\end{lema}

\begin{proof}
We write $\alpha e_{i,i+k-1}=p(e_{i_{1},j_{1}},\dots,e_{i_{m},j_{m}})$ for some nonzero $\alpha\in\K$. Hence $i+k-1=j_{l}$ for some indexes $l\in\{1,\dots,m\}$. Replacing for each $l$ the corresponding $j_{l}$ by $j_{l}+1$ we get $$\alpha e_{i,i+k}=p(e_{i_{1},j_{1}},\dots,e_{i_{l},j_{l}+1},\dots,e_{i_{m},j_{m}})$$ which proves that $e_{i,i+k}\in Im(p)$. 
\end{proof}

If we also denote $UT_{n}$ by $UT_{n}^{(-1)}$, then we have the main result of this section.

\begin{prp}
Let $p$ be a multilinear polynomial over $\K$. Then $\langle Im(p) \rangle$ is either $\{0\}$ or $UT_{n}^{(k)}$ for some integer $k\geq -1$.
\end{prp}

\begin{proof}
Assume that $Im(p)$ is nonzero. Hence, if $A=\displaystyle\sum_{i,j=1}^{n}a_{ij}e_{ij}\in Im(p)$ is nonzero, writing $A$ as a linear combination of evaluations of $p$ on upper triangular matrix units, we get that a multiple of $e_{ij}$ belongs to $Im(p)$, for each nonzero $(i,j)$ entry of $A$.

Let $k$ be the minimal integer such that the $k$-th diagonal of some matrix $A=\displaystyle\sum_{i,j=1}^{n}a_{ij}e_{ij}\in Im(p)$ is nonzero. Then there exists some $a_{i,i+k-1}\neq0$ and therefore $\alpha e_{i,i+k-1}=p(e_{i_{1},j_{1}},\dots,e_{i_{m},j_{m}})$ for some nonzero $\alpha\in\K$. By Lemma \ref{l2} all the matrix units  $e_{1,k},\dots,e_{n-k+1,n}$ belong to $Im(p)$. By Lemma \ref{l4} $e_{i,i+k}\in Im(p)$. Using these both lemmas alternatively, we get that $UT_{n}^{(k-2)}\subset \langle Im(p) \rangle$. By the minimality of $k$ we have $\langle Im(p) \rangle = UT_{n}^{(k-2)}$.
\end{proof}

By the above proposition we can restate Conjecture \ref{c1} as

\begin{con}
The image of a multilinear polynomial on the upper triangular matrix algebra is either $\{0\}$ or $UT_{n}^{(k)}$ for some integer $k\geq-1$. 
\end{con}

\section{A technical proposition}

We start with a fact about the image of multilinear polynomials of any degree on $UT_{n}$. We will prove that no subset between $UT_{n}^{(0)}$ and $UT_{n}$ can be the image of a multilinear polynomial over $UT_{n}$. 

\begin{prp}\label{p1}
Let $\K$ be any field, $m\geq2$ an integer and $$p(x_{1},\dots,x_{m})=\sum_{\sigma\in S_{m}}\alpha_{\sigma}x_{\sigma(1)}\cdots x_{\sigma(m)},\alpha_{\sigma}\in\K,$$ a nonzero multilinear polynomial. 

\begin{itemize}
\item[$(1)$] if $\displaystyle\sum_{\sigma\in S_{m}}\alpha_{\sigma}\neq0$, then $Im(p)=UT_{n}$;
\item[$(2)$] if $\displaystyle\sum_{\sigma\in S_{m}}\alpha_{\sigma}=0$ and $UT_{n}^{(0)}\subset Im(p)$, then $Im(p)=UT_{n}^{(0)}$.
\end{itemize}
\end{prp}

\begin{proof}
If $\displaystyle\sum_{\sigma\in S_{m}}\alpha_{\sigma}\neq0$, then replacing $m-1$ variables by $I_{n}$ (the identity matrix) and the last one by $(\displaystyle\sum_{\sigma\in S_{m}}\alpha_{\sigma})^{-1}A$ where $A$ is any matrix in $UT_{n}$, we get $Im(p)=UT_{n}$, from which $(1)$ follows.

If $\displaystyle\sum_{\sigma\in S_{m}}\alpha_{\sigma}=0$ and $UT_{n}^{(0)}\subset Im(p)$, then let $\tau\in S_{n}$ such that $\alpha_{\tau}\neq0$ (there exists such a permutation because $p\neq0$). Then, $\alpha_{\tau}=-\displaystyle\sum_{\sigma \in S_{m}\setminus \{\tau\}}\alpha_{\sigma}.$

So, $$p(x_{1},\dots,x_{m})=\sum_{\sigma\in S_{m}\setminus\{\tau\}} \alpha_{\sigma}(x_{\sigma(1)}\cdots x_{\sigma(m)}-x_{\tau(1)}\cdots x_{\tau(m)}).$$

Therefore, replacing $x_{1},\dots,x_{m}$ by upper triangular matrices we obtain in each term of the sum above a matrix with just zeros in the main diagonal. Indeed, the main diagonal of a product of upper triangular matrices is the same, regardless of the order.

With this, we conclude that $Im(p)\subset UT_{n}^{(0)}$ and by hypothesis, $Im(p)=UT_{n}^{(0)}$.
\end{proof}

\section{The images of multilinear polynomials of degree two}

We consider a multilinear polynomial of degree two, which has the following form: $p(x,y)=\alpha xy+\beta yx$ for some $\alpha,\beta\in\K$. We will divide the study of the image of $p$ in two cases.

Case 1. $\alpha+\beta\neq0$.

In this case we can use Proposition \ref{p1} (1) and get $Im(p)=UT_{n}$.

Case 2. $\alpha+\beta=0$.

If $\alpha=\beta=0$ then $Im(p)=\{0\}$. Otherwise, we may assume that $p(x,y)=xy-yx$. Let $A=(a_{ij})\in UT_{n}^{(0)}$. Take $B=\displaystyle\sum_{k=1}^{n-1}e_{k,k+1}$ and $C=(c_{ij})\in UT_{n}$. So,
\begin{eqnarray}\label{comutador}
BC-CB&=&(\sum_{k=1}^{n-1}e_{k,k+1})(\sum_{i,j=1}^{n}c_{ij}e_{ij})-(\sum_{i,j=1}^{n}c_{ij}e_{ij})(\sum_{k=1}^{n-1}e_{k,k+1})\\\nonumber
 &=& \sum_{i=1}^{n-1}\sum_{j=2}^{n}(c_{i+1,j}-c_{i,j-1})e_{ij}
\end{eqnarray}

Using $c_{ij}=0$ for $i>j$ , we note that the diagonal entries of the matrix $BC-CB$ above are all zero.

Now we consider the system defined by the equations $c_{i+1,j}-c_{i,j-1}=a_{ij}$. A solution of this system is $c_{1k}=0,k=1,\dots,n$ and $c_{i+1,j}=a_{ij}+a_{i-1,j-1}+\cdots+a_{1,j-(i-1)}$ where $i<j$ and $i=2,\dots,n-1,j=2,\dots,n$.

So, $Im(p)\supset UT_{n}^{(0)}$ and by Proposition \ref{p1} $(2)$, we have $Im(p)=UT_{n}^{(0)}$.

In resume, we have proved the following

\begin{prp}\label{4}
Let $p(x,y)\in\K\langle X \rangle$ be a multilinear polynomial where $\K$ is any field. Then $Im(p)$ on $UT_{n}$ is $\{0\}, UT_{n}^{(0)}$ or $UT_{n}$.
\end{prp}

\section{The images of multilinear polynomials of degree three}

To start this section we prove the following lemma, which is a an analogous of Lemma 1.2 of \cite{Amitsur}.

\begin{lema}\label{l1}
Let $\K$ be a field with at least $n$ elements and let $d_{11},\dots,d_{nn}\in\K$ be distinct elements. Then for $D=diag(d_{11},\dots,d_{nn})$ and $k\geq 0$, we have 
\begin{eqnarray}\nonumber
[UT_{n}^{(k)},D]=UT_{n}^{(k)}  \ \mbox{and}\  [UT_{n},D]=UT_{n}^{(0)}.
\end{eqnarray}
\end{lema}

\begin{proof}	Clearly, $[UT_{n}^{(k)},D]\subset UT_{n}^{(k)}$.
	
	Now, let $A=\displaystyle\sum_{j-i>k}a_{ij}e_{ij}$ be an arbitrary element of $UT_{n}^{(k)}$. Then,
	\begin{eqnarray}\nonumber
	[A,D]&=&AD-DA=(\sum_{i,j=1}^{n}a_{ij}e_{ij})(\sum_{l=1}^{n}d_{ll}e_{ll})-(\sum_{l=1}^{n}d_{ll}e_{ll})(\sum_{i,j=1}^{n}a_{ij}e_{ij})\\\nonumber
	&=&\sum_{j-i>k}^{n}a_{ij}(d_{jj}-d_{ii})e_{ij}
	\end{eqnarray}
	
	Hence, if $B=\displaystyle\sum_{j-i>k}^{n}b_{ij}e_{ij}\in UT_{n}^{(k)}$, we choose $a_{ij}=b_{ij}(d_{jj}-d_{ii})^{-1},$ for $j-i>k$. This proves that $[UT_n^{(k)}, D]\supset UT_n^{(k)}$, and the first equality is proved. 
	
	Now we prove the second equality. It is immediate that $[UT_n,D]\subset UT_n^{(0)}$. Also, since $UT_n^{(0)}\subset UT_n$, we have $[UT_n^{(0)},D]\subset [UT_n,D]$. Hence, from the first equation for $k=0$, we have $UT_n^{(0)}=[UT_n^{(0)},D]\subset [UT_n,D]$. And the second equality is proved.	
%	
%Clearly $[UT_{n},D]\subset UT_{n}^{(0)}$. Reciprocally, let $B=(b_{ij})\in UT_{n}$. Note that 
%\begin{eqnarray}\nonumber
%[B,D]&=&BD-DB=(\sum_{i,j=1}^{n}b_{ij}e_{ij})(\sum_{k=1}^{n}d_{kk}e_{kk})-(\sum_{k=1}^{n}d_{kk}e_{kk})(\sum_{i,j=1}^{n}b_{ij}e_{ij})\\\nonumber
%&=& \sum_{i,j=1}^{n}b_{ij}(d_{jj}-d_{ii})e_{ij}
%\end{eqnarray}
%So, for $A=(a_{ij})\in UT_{n}^{(0)}$ we may take $b_{ij}=a_{ij}(d_{jj}-d_{ii})^{-1}$ when $i<j$, and then $A\in [UT_{n},D]$.
%
%The other equality has an analogous proof.
\end{proof}

Following the proof of Theorem 13 of \cite{Mesyan}, we obtain the next theorem, where we determine the image of multilinear polynomials of degree 3 on $UT_{n}$.

\begin{teor}\label{6}
Let $\K$ be a field with at least $n$ elements and let $p(x,y,z)\in \K\langle X \rangle$ be a multilinear polynomial. Then $Im(p)$ is $\{0\}, UT_{n}^{(0)}$ or $UT_{n}$.
\end{teor}

\begin{proof}
Let $p(x,y,z)\in\K\langle X \rangle$ be a nonzero multilinear polynomial. So, $$p(x,y,z)=\alpha_{1}xyz+\alpha_{2}xzy+\alpha_{3}yxz+\alpha_{4}yzx+\alpha_{5}zxy+\alpha_{6}zyx,\alpha_{l}\in\K.$$ If $\alpha_{1}+\alpha_{2}+\alpha_{3}+\alpha_{4}+\alpha_{5}+\alpha_{6}\neq0$ then using Proposition \ref{p1} (1) we have $Im(p)=UT_{n}$.

Hence, we may assume that $\alpha_{1}+\alpha_{2}+\alpha_{3}+\alpha_{4}+\alpha_{5}+\alpha_{6}=0$. So, we write $p$ as $$p(x,y,z)=\alpha_{1}(xyz-zyx)+\alpha_{2}(xzy-zyx)+\alpha_{3}(yxz-zyx)+\alpha_{4}(yzx-zyx)+\alpha_{5}(zxy-zyx).$$

If any of $p(1,y,z),p(x,1,z)$ or $p(x,y,1)$ are non-zero, then we have by Proposition \ref{4} that $Im(p)$ contains all upper triangular matrices with zero main diagonal. Then by Proposition \ref{p1} we have that $Im(p)$ is $UT_{n}^{(0)}$ or $UT_{n}$.

Otherwise, the equations $p(1,y,z)=p(x,1,z)=p(x,y,1)=0$ imply that $\alpha_{3}=\alpha_{5},\alpha_{2}=\alpha_{4}$ and $\alpha_{1}=-\alpha_{2}-\alpha_{3}$. Therefore, 
\begin{eqnarray}\nonumber
p(x,y,z)&=&(-\alpha_{2}-\alpha_{3})(xyz-zyx)+\alpha_{2}(xzy-zyx+yzx-zyx)\\\nonumber
& & + \ \alpha_{3}(yxz-zyx+yzx-xyz)\\\nonumber
&=& \alpha_{2}(xyz-zyx+yzx-xyz)+\alpha_{3}(yxz-zyx+zxy-xyz)\\\nonumber
&=& \alpha_{2}[x,[z,y]]+\alpha_{3}[z,[x,y]]
\end{eqnarray}

Since $p\neq0$, renaming the variables if necessary, we may assume that $\alpha_{2}\neq0$ and therefore assume $$p(x,y,z)=[x,[z,y]]+\alpha[z,[x,y]],$$ for some $\alpha\in\K$.

By Lemma \ref{l1}, $UT_{n}^{(0)}=[D,UT_{n}^{(0)}]=[D,[UT_{n},D]]$. So, taking $x=y=D$ and $z=A$ any matrix in $UT_{n}$ we get all of $UT_{n}^{(0)}$. So, $Im(p)=UT_{n}^{(0)}$.
\end{proof}

\section{The images of multilinear polynomials of degree four}

In this section we will determinate the image of multilinear polynomials of degree four over a field $\K$ of zero characteristic. 

We start with the following lemma.
%
%\begin{lema}\label{l2}
%Let $\K$ be an infinite field and $D=diag(d_{11},\dots,d_{nn})$ with $d_{11},\dots,d_{nn}\in\K$ distinct elements. Then, $[UT_{n}^{(1)},D]=UT_{n}^{(1)}$.
%\end{lema}
%
%\begin{proof}
%Clearly, $[UT_{n}^{(1)},D]\subset UT_{n}^{(1)}$.
%
%Now, let $A=\displaystyle\sum_{i,j=1}^{n}a_{ij}e_{ij}\in UT_{n}^{(1)}$. Then,
%\begin{eqnarray}\nonumber
%[A,D]&=&AD-DA=(\sum_{i,j=1}^{n}a_{ij}e_{ij})(\sum_{k=1}^{n}d_{kk}e_{kk})-(\sum_{k=1}^{n}d_{kk}e_{kk})(\sum_{i,j=1}^{n}a_{ij}e_{ij})\\\nonumber
%&=&\sum_{i,j=1}^{n}a_{ij}(d_{jj}-d_{ii})e_{ij}
%\end{eqnarray}
%
%Hence, for $B=\displaystyle\sum_{i,j=1}^{n}b_{ij}e_{ij}\in UT_{n}^{(1)}$, we choose $a_{ij}=b_{ij}(d_{jj}-d_{ii})^{-1},$ for $i<j$.
%\end{proof}

\begin{lema}\label{l3} %Esse era o Lemma8 antes da mudança
Let $\K$ be any field. Then $[UT_{n}^{(0)},UT_{n}^{(0)}]=UT_{n}^{(1)}$.
\end{lema}

\begin{proof}
Clearly, $[UT_{n}^{(0)},UT_{n}^{(0)}]\subset UT_{n}^{(1)}$. 

Now, let $A=\displaystyle\sum_{k=1}^{n}e_{k,k+1}\in UT_{n}^{(0)}$ and $B=\displaystyle\sum_{i,j=1}^{n}b_{ij}e_{ij}\in UT_{n}^{(0)}$. The same computations as in equation (\ref{comutador}) yields
\begin{eqnarray}\nonumber
[A,B]&=&%AB-BA=(\sum_{k=1}^{n}e_{k,k+1})(\sum_{i,j=1}^{n}b_{ij}e_{ij})-(\sum_{i,j=1}^{n}b_{ij}e_{ij})(\sum_{k=1}^{n}e_{k,k+1})\\\nonumber
%&=& 
\sum_{i=1}^{n-1}\sum_{j=2}^{n}(b_{i+1,j}-b_{i,j-1})e_{ij}.
\end{eqnarray}
So, for $C=(c_{ij})\in UT_{n}^{(1)}$, the system below has solution
\begin{eqnarray}
\left\{\begin{array}{ccc}\nonumber
b_{23}-b_{12}&=&c_{13}\\
&\vdots &\\
b_{2n}-b_{1,n-1}&=&c_{1n}\\
&\vdots &\\
b_{n-1,n}-b_{n-2,n-1}&=&c_{n-2,n}\\
\end{array}\right..
\end{eqnarray}
Indeed, we may choose $b_{1k}=0, k=2,\dots,n-1$ and $b_{i+1,j}=c_{i,j}+\cdots+c_{1,j-(i-1)},$ $i=1,\dots,n-2,j=3,\dots,n$. Therefore, $C\in [UT_{n}^{(0)},UT_{n}^{(0)}]$.
\end{proof}

Now we prove the main result for polynomials of degree 4. Ou proof is based on the proof of Theorem 1 of \cite{Buzinski}.

\begin{teor}
Let $\K$ be a field of zero characteristic and let $p(x_{1},x_{2},x_{3},x_{4})\in\K\langle X \rangle$ be  a multilinear polynomial. Then the image of $p$ on $UT_{n}$ is $\{0\}, UT_{n}^{(1)},UT_{n}^{(0)}$ or $UT_{n}$.
\end{teor}

\begin{proof}
We may assume that $p\neq0$. If any of $p(1,x_2,x_3,x_4)$, $p(x_1,1,x_3,x_4)$, $p(x_1,x_2,1,x_4)$ or $p(x_1,x_2,x_3,1)$ are nonzero, then by Proposition \ref{p1} and Theorem \ref{6}, we have $Im(p)=UT_{n}^{(0)}$ or $UT_{n}$. So, we may assume that $$ p(1,x_2,x_3,x_4)=p(x_1,1,x_3,x_4)=p(x_1,x_2,1,x_4)=p(x_1,x_2,x_3,1)=0.$$ Then by Falk's theorem \cite{Falk} we have 
\begin{eqnarray}\nonumber
p(x_{1},x_{2},x_{3},x_{4})&=&L(x_{1},x_{2},x_{3},x_{4})+\alpha_{1234}[x_{1},x_{2}][x_{3},x_{4}]+\alpha_{1324}[x_{1},x_{3}][x_{2},x_{4}]\\\nonumber
& &+\alpha_{1423}[x_{1},x_{4}][x_{2},x_{3}]+\alpha_{2314}[x_{2},x_{3}][x_{1},x_{4}]+\alpha_{2413}[x_{2},x_{4}][x_{1},x_{3}]\\\nonumber
& &+\alpha_{3412}[x_{3},x_{4}][x_{1},x_{2}]
\end{eqnarray}

where $L(x_{1},x_{2},x_{3},x_{4})$ is a Lie polynomial and $\alpha_{1234},\alpha_{1324},\alpha_{1423},\alpha_{2314},\alpha_{2413},\alpha_{3412} \in \K$.

Using Hall basis (see \cite{Hall}) we can write 
\begin{eqnarray}\nonumber
L(x_{1},x_{2},x_{3},x_{4})&=&\alpha_{1}[[[x_{2},x_{1}],x_{3}],x_{4}]+\alpha_{2}[[[x_{3},x_{1}],x_{2}],x_{4}]+\alpha_{3}[[[x_{4},x_{1}],x_{2}],x_{3}]\\\nonumber
& & +\alpha_{4}[[x_{4},x_{1}],[x_{3},x_{2}]]+\alpha_{5}[[x_{4},x_{2}],[x_{3},x_{1}]]+\alpha_{6}[[x_{4},x_{3}],[x_{2},x_{1}]],
\end{eqnarray}
where $\alpha_{1},\alpha_{2},\alpha_{3},\alpha_{4},\alpha_{5},\alpha_{6}\in\K$.

Opening the brackets for the three last terms we can assume $p$ as
\begin{eqnarray}\nonumber
p(x_{1},x_{2},x_{3},x_{4})&=&\alpha_{1}[[[x_{2},x_{1}],x_{3}],x_{4}]+\alpha_{2}[[[x_{3},x_{1}],x_{2}],x_{4}]+\alpha_{3}[[[x_{4},x_{1}],x_{2}],x_{3}]\\\nonumber
& & +\alpha_{1234}[x_{1},x_{2}][x_{3},x_{4}]+\alpha_{1324}[x_{1},x_{3}][x_{2},x_{4}]+\alpha_{1423}[x_{1},x_{4}][x_{2},x_{3}]\\\nonumber
& & +\alpha_{2314}[x_{2},x_{3}][x_{1},x_{4}]+\alpha_{2413}[x_{2},x_{4}][x_{1},x_{3}] +\alpha_{3412}[x_{3},x_{4}][x_{1},x_{2}].
\end{eqnarray}

Now suppose that for some $i=1,2,3$ we have $\alpha_{i}\neq0$. Without loss of generality, we may assume that $\alpha_{1}\neq0$. So, replacing $x_{1},x_{3}$ and $x_{4}$ by $D=diag(d_{11},\dots,d_{nn})$ with $d_{11},\dots,d_{nn}$  distinct elements in $\K$, we get $p(D,x_{2},D,D)=\alpha_{1}[[[x_{2},D],D],D]$. Now, using Lemma \ref{l1} we have $Im(p)=UT_{n}^{(0)}$. So, we may assume that $\alpha_{1}=\alpha_{2}=\alpha_{3}=0$ and then 
\begin{eqnarray}\nonumber
p(x_{1},x_{2},x_{3},x_{4})&=&\alpha_{1234}[x_{1},x_{2}][x_{3},x_{4}]+\alpha_{1324}[x_{1},x_{3}][x_{2},x_{4}]+\alpha_{1423}[x_{1},x_{4}][x_{2},x_{3}]\\\nonumber
& & +\alpha_{2314}[x_{2},x_{3}][x_{1},x_{4}]+\alpha_{2413}[x_{2},x_{4}][x_{1},x_{3}]+\alpha_{3412}[x_{3},x_{4}][x_{1},x_{2}]. 
\end{eqnarray}

Clearly, $Im(p)\subset UT_{n}^{(1)}$.

We will consider two cases.

Case 1. Assume $\alpha_{1234}=\alpha_{2314}=\alpha_{3412}=\alpha_{1423}=-\alpha_{1324}=-\alpha_{2413}$. Then we may assume that
\begin{eqnarray}\nonumber
p(x_{1},x_{2},x_{3},x_{4})&=&[x_{1},x_{2}][x_{3},x_{4}]+[x_{3},x_{4}][x_{1},x_{2}]+[x_{2},x_{3}][x_{1},x_{4}]+[x_{1},x_{4}][x_{2},x_{3}]\\\nonumber
& & -[x_{1},x_{3}][x_{2},x_{4}]-[x_{2},x_{4}][x_{1},x_{3}].
\end{eqnarray}

Consider $A\in UT_{n}^{(1)}$. Let $D=diag(d_{11},\dots,d_{nn})$ where $d_{11},\dots,d_{nn}$ are all distinct elements of $\K$. Then, by Lemma \ref{l1} there exists $G\in UT_{n}^{(1)}$ with $A=[D,G]$. By Lemma \ref{l3} there are $E,F\in UT_{n}^{(0)}$ such that $G=[E,F]$. Again by Lemma \ref{l1} we have $B,C\in UT_{n}$ such that $E=[D,B]$ and $F=[D,C]$. So, $A=[D,[[D,B],[D,C]]]$. Observing that
\begin{eqnarray}\nonumber
p(D,D^{2},B,C)&=&[D,D^{2}][B,C]+[B,C][D,D^{2}]+[D^{2},B][D,C]+[D,C][D^{2},B]\\\nonumber
& &-[D,B][D^{2},C]-[D^{2},C][D,B]\\\nonumber
&=&[D^{2},B][D,C]+[D,C][D^{2},B]-[D,B][D^{2},C]-[D^2,C][D,B]\\\nonumber
&=&[D,[[D,B],[D,C]]],
\end{eqnarray}

we have $A\in Im(p)$, proving in this way that $Im(p)=UT_{n}^{(1)}$.

Case 2. Assume that at least one of following $\alpha_{1234}=\alpha_{2314}=\alpha_{3412}=\alpha_{1423}=-\alpha_{1324}=-\alpha_{2413}$ does not hold. So, there are $A,B,C \in UT_{n}$ such that at least one of the following expressions is not zero: 
\begin{eqnarray}\nonumber
p(A,A,B,C)&=&(\alpha_{1324}+\alpha_{2314})[A,B][A,C]+(\alpha_{1423}+\alpha_{2413})[A,C][A,B],\\\nonumber
p(A,B,A,C)&=&(\alpha_{1234}-\alpha_{2314})[A,B][A,C]+(\alpha_{3412}-\alpha_{1423})[A,C][A,B],\\\nonumber
p(A,B,C,A)&=&(-\alpha_{1234}-\alpha_{2413})[A,B][A,C]+(-\alpha_{1324}-\alpha_{3412})[A,C][A,B],\\\nonumber
p(B,A,A,C)&=&(-\alpha_{1234}-\alpha_{1324})[A,B][A,C]+(-\alpha_{2413}-\alpha_{3412})[A,C][A,B],\\\nonumber
p(B,A,C,A)&=&(-\alpha_{1423}+\alpha_{1234})[A,B][A,C]+(\alpha_{3412}+\alpha_{2314})[A,C][A,B],\\\nonumber
p(B,C,A,A)&=&(\alpha_{1324}+\alpha_{1423})[A,B][A,C]+(\alpha_{2314}+\alpha_{2413})[A,C][A,B].
\end{eqnarray}

Therefore, we may reduce the problem to prove that with the expression $$[A,B][A,C]+\lambda [A,C][A,B], \lambda \in \K,$$ we get all elements in $UT_{n}^{(1)}$. Using Lemma \ref{l1} and taking $A=diag(a_{11},\dots,a_{nn})$ where all $a_{11},\dots,a_{nn}$ are distinct elements of $\K$, there exist $B\in UT_{n}$ such that $\displaystyle\sum_{k=1}^{n-1}e_{k,k+1}=[A,B]$. Writing $[A,C]=\displaystyle\sum_{i,j=1}^{n}b_{ij}e_{ij}$ we have
\begin{eqnarray}\nonumber
[A,B][A,C]+\lambda [A,C][A,B]&=&(\sum_{k=1}^{n-1}e_{k,k+1})(\sum_{i,j=1}^{n}b_{ij}e_{ij})+\lambda(\sum_{i,j=1}^{n}b_{ij}e_{ij})(\sum_{k=1}^{n-1}e_{k,k+1})\\\nonumber
&=&\sum_{i=1}^{n-1}\sum_{j=2}^{n}(b_{i+1,j}+\lambda b_{i,j-1})e_{ij}
\end{eqnarray}

So, for an arbitrary $M=(c_{ij})\in UT_{n}^{(1)}$, %D não é uma boa letra.
 the system below has solution.
\begin{eqnarray}
\left\{\begin{array}{ccc}\nonumber
b_{23}+\lambda b_{12}&=&c_{13}\\
&\vdots &\\
b_{2n}+\lambda b_{1,n-1}&=&c_{1n}\\
&\vdots &\\
b_{n-1,n}+\lambda b_{n-2,n-1}&=&c_{n-2,n}\\
\end{array}\right..
\end{eqnarray}
Indeed, we may choose $b_{1k}=0, k=2,\dots,n-1$ and $$b_{i+1,j}=c_{i,j}-\lambda c_{i-1,j-1}+\cdots+(-\lambda)^{i-1}c_{1,j-(i-1)},i=1,\dots,n-2,j=3,\dots,n.$$ 

Therefore, $M\in Im(p)$, proving that $Im(p)=UT_{n}^{(1)}$.
\end{proof}

\section*{Acknowledgments}

This work was completed when the first author visited Kent State University. The first author would like to thank the Department of Mathematical Sciences of Kent State University for its hospitality. The authors would like to thank Dr. Mikhail Chebotar for the helpful comments and the anonymous referee for his/her useful suggestions that much improved the final version of this paper.

\section*{Funding}

The first author was supported by São Paulo Research Foundation (FAPESP), grants \# 2017/16864-5 and \# 2016/09496-7.
The second author was supported by São Paulo Research Foundation (FAPESP), grant \# 2014/09310-5 and by National Council for Scientific and Technological Development (CNPq), grant \# 461820/2014-5.


\begin{thebibliography}{00}

\bibitem{Albert} A. A. Albert and B. Muckenhoupt, {\it On matrices of trace zero}, Michigan Math. J. {\bf 4} (1957), 1--3.

\bibitem{Amitsur} S. A. Amitsur and L. Rowen, {\it Elements of reduced trace 0}, Israel J. Math. {\bf 87} (1994), 161--179.

\bibitem{Buzinski} D. Buzinski and R. Winstanley, {\it On multilinear polynomials in four variables evaluated on matrices}, Linear Algebra Appl. {\bf 439} (2013), 2712--2719.

\bibitem{Dniester} {\it Dniester notebook: unsolved problems in the theory of rings and modules}, Non-associative algebra and its applications, Lect. Notes Pure Appl. Math., vol. 246, Chapman \& Hall/CRC, Boca Raton, FL, 2006, pp. 461--516. Translated from the 1993 Russian edition by Murray R. Bremnerand Mikhail V. Kochetov and edited by V. T. Filippov, V. K. Kharchenko and I. P. Shestakov.

\bibitem{Falk} G. Falk, {\it Konstanzelemente in Ringen mit Differentiation}, Math. Ann. {\bf 124} (1952), 182–186.

\bibitem{Hall} M. Hall, {\it A basis for free Lie rings and higher commutators in free groups}, Proc. Amer. Math. Soc. {\bf 1} (1950), 575–581.

\bibitem{Kanel2} A. Kanel-Belov, S. Malev and L. Rowen, {\it The images of non-commutative polynomials evaluated on $2\times 2$ matrices}, Proc. Amer. Math. Soc. {\bf 140} (2012), 465--478.

\bibitem{Mesyan} Z. Mesyan, {\it Polynomials of small degree evaluated on matrices}, Linear Multilin. Alg. {\bf 61} (2013), 1487--1495.

\bibitem{Shoda}
K. Shoda, {\it Einige Sätze über Matrizen}, Jap. J. Math. {\bf 13} (1936), 361--365.

\bibitem{Spela}
\v{S}. \v{S}penko, {\it On the image of a noncommutative polynomial}, J. Algebra {\bf 377} (2013), 298--311.


\end{thebibliography}
\end{document}